\documentclass[10pt]{article}

\usepackage[top=2.5cm, bottom=2.5cm, left=2.5cm, right=2.5cm, includefoot]{geometry}

\usepackage[applemac]{inputenc}
\usepackage{amssymb,amsmath,amsthm,bbm}
\usepackage{graphicx,color}
\usepackage{url}

\usepackage{authblk}

\newtheorem{thm}{Theorem}[section]

\allowdisplaybreaks

\newcommand{\al}{\alpha}
\newcommand{\p}{\partial}

\newcommand{\sig}{\sigma}

\newcommand{\Gam}{\Gamma}
\newcommand{\lam}{\lambda}
\newcommand{\D}{\Delta}
\newcommand{\ep}{\epsilon}
\newcommand{\del}{\delta}
\newcommand{\vp}{\varphi}

\newcommand{\tes}{\hat{\theta}_{n}}
\newcommand{\bes}{\hat{\beta}_{n}}

\newcommand{\ses}{\hat{\sig}_{n}}
\newcommand{\tz}{\theta_{0}}

\newcommand{\res}{\hat{\rho}_{n}}

\newcommand{\E}{\mathbb{E}}
\newcommand{\pr}{\mathbb{P}}
\newcommand{\cil}{\Rightarrow}
\newcommand{\cip}{\xrightarrow{p}}

\newcommand{\mbbg}{\mathbb{G}}

\newcommand{\mbbn}{\mathbb{N}}
\newcommand{\mbbr}{\mathbb{R}}
\newcommand{\mbbrp}{\mathbb{R}_{+}}
\newcommand{\mbbs}{\mathbb{S}}
\newcommand{\mcc}{\mathcal{C}}
\newcommand{\mcl}{\mathcal{L}}
\newcommand{\mci}{\mathcal{I}}

\def\ds#1{\displaystyle{#1}}

\def\nn{\nonumber}

\def\tcr#1{\textcolor{red}{#1}}

\def\var{{\rm var}}

\def\sumj{\sum_{j=1}^{n}}
\def\lp{L\'evy process}

\title{Efficient estimation of stable L\'{e}vy process with symmetric jumps
\thanks{This work was partially supported by JSPS KAKENHI Grant Number JP26400204 and JST CREST Grant Number JPMJCR14D7, Japan (HM).
}
}

\author[1]{Alexandre Brouste\thanks{Avenue Olivier Messiaen, 72085 Le Mans Cedex 9, France. Email: {\tt Alexandre.Brouste@univ-lemans.fr}}}
\author[2]{Hiroki Masuda\thanks{744 Motooka Nishi-ku Fukuoka, 819-0395, Japan. Email: {\tt hiroki@math.kyushu-u.ac.jp} (Corresponding author)}}

\affil[1]{Laboratoire Manceau de Math\'{e}matiques, Le Mans Universit\'{e}, France}
\affil[2]{Faculty of Mathematics, Kyushu University, Japan}

\begin{document}

\maketitle

\begin{abstract}
Efficient estimation of a non-Gaussian stable L\'{e}vy process with drift and symmetric jumps observed at high frequency is considered. For this statistical experiment, the local asymptotic normality of the likelihood is proved with a non-singular Fisher information matrix through the use of a non-diagonal norming matrix. The asymptotic normality and efficiency of a sequence of roots of the associated likelihood equation are shown as well. Moreover, we show that a simple preliminary method of moments can be used as an initial estimator of a scoring procedure, thereby conveniently enabling us to bypass numerically demanding likelihood optimization. Our simulation results show that the one-step estimator can exhibit quite similar finite-sample performance as the maximum likelihood estimator.
\end{abstract}

\section{Introduction}


Let $(X_{t})_{t\ge 0}$ be a $\beta$-stable {\lp} with a drift $\mu\in \mbbr$ and symmetric jumps defined by
\begin{equation}
X_{t} = \mu t + \sig J_{t}, \quad t \geq 0,
\nonumber
\end{equation}
where $(J_t)_{t\ge 0}$ denotes the standard symmetric $\beta$-stable {\lp} characterized by its characteristic function
\begin{equation}
\E(e^{iu J_1}) = e^{- |u|^\beta}, \qquad u\in\mbbr.
\nonumber
\end{equation}
We refer to \cite{Sat99}, \cite{Zol86} for a comprehensive account of the stable distribution.
Throughout, we assume that the process $(X_{t})_{t\ge 0}$ is observed on regularly spaced time points $t_{j}^{n}=jh_{n}$ for $j=0,1,\dots,n$ in the high-frequency setting with sampling stepsize $h_{n}\to 0$ as $n \to \infty$ and with the terminal sampling time $nh_n$ satisfying 
\begin{equation}
\liminf_{n}nh_n >0.
\label{hm:liminf.p}
\end{equation}
The asymptotically efficient estimation of the three-dimensional unknown parameter
\begin{equation}
\theta := (\beta,\sig,\mu) \in (0,2)\times(0,\infty)\times\mbbr
\nonumber
\end{equation}
is considered in this paper.

As was shown in \cite{AitJac08} and \cite{Mas15LM}, the joint estimation of the scale $\sigma$ and stable (or self-similarity) index $\beta$ leads to a singular Fisher information matrix as long as a diagonal-matrix norming rate is used.
Consequently, the conventional convolution and minimax theorems are not of direct use, and an asymptotically minimal covariance matrix of sequences of estimators  cannot be deduced in this statistical experiment.
Any proper notion of efficiency for this experiment has not been detailed as yet in the literature, despite the theoretical and practical importance of the high-frequency data analysis and the fact that the stable {\lp es} constitute a fundamental class of non-Gaussian {\lp es}.

The asymptotic degeneracy is an intrinsic feature coming from the self-similarity of the underlying model. Recently, in the context of the observation of fractional Brownian motion observed at high-frequency, the singularity issue has been untied in \cite{BroFuk16} by using a non-diagonal norming matrix, and a classical local asymptotic normality (LAN) property of the likelihoods has been obtained with a non-degenerate Fisher information. The associated H\'{a}jek-Le Cam asymptotic minimax theorem is derived as a corollary; we refer to \cite{IbrHas81} for a detailed account of the LAN property and its consequences. It is worth mentioning that a non-diagonal norming has also been considered in proving the LAN property for the drift and scale parameters of a possibly skewed locally stable {\lp} from high-frequency data, when the activity index is assumed to be known (see \cite[Remark 2.2]{IvaKulMas15} for some related remarks),
and in the joint estimation of the index and the scale parameters of a stable subordinator when observing the
biggest $n$ jump sizes with their jump instants precisely (see \cite{HopJac94}).

Contrary to the large sample asymptotics with fixed sampling stepsize,
 the LAN property of the likelihoods cannot be deduced from Le Cam's second lemma in the high-frequency setting. Consequently, this study enlarges the domain of application of the LAN theory already proved for i.i.d. classical setting \cite{Haj72}, \cite{LeC72}, ergodic Markov chains \cite{Rou72}, ergodic diffusions \cite{Gob02}, diffusions under high-frequency observations \cite{Gob01},  diffusions with observational noise \cite{GloJac01-1}, \cite{GloJac01-2}, \cite{Rei11}, several L\'evy process or L\'{e}vy driven models \cite{CleGlo15}, \cite{KawMas13}, \cite{Mas15LM} and fractional Gaussian noise \cite{BroFuk16}, \cite{CohGamLacLou13}.

\medskip

The purpose of this paper is twofold. Firstly, the LAN property of the likelihoods for the aforementioned high-frequency statistical experiment is proved with a non-diagonal norming rate and a non-singular Fisher information matrix.
In particular, a H\'{a}jek-Le Cam lower bound for any estimator can be derived so that efficiency can be described. A sequence of maximum likelihood estimators (MLE) is shown to be asymptotically normal and asymptotically efficient.
In the proofs, the analytic properties of the probability density function of $J_1$ is fully used.

Although the exact likelihood function can be numerically computed, the likelihood function considered here involves the local scaling by the factor $h_n^{-1/\beta}$,
making numerical evaluation of a maximum-likelihood estimate more time-consuming compared with the i.i.d. model setup (see \eqref{hm:log-lf} below). In this respect, it is beneficial to construct a computationally easy-to-use estimator which is asymptotically equivalent to the MLE.
Therefore, as a second purpose of the present study, we exemplify such an asymptotically optimal estimator through a preliminary method of moments combined with the scoring procedure.

\medskip

The LAN property of the likelihoods and the asymptotic normality of the sequence of MLE are described in Section \ref{hm:sec_lfa}, followed by asymptotic efficiency of the MLE.
Construction of an asymptotically optimal simple estimator is detailed in Section \ref{hm:sec_optest} with some numerical experiments.

\section{Likelihood asymptotics}\label{hm:sec_lfa}


We write $\to_{u}$ for the ordinary uniform convergence (for non-random quantities) with respect to $\theta$ over any compact set contained in $(0,2)\times(0,\infty)\times\mbbr$.
For any continuous random functions $\xi_{0}(\theta)$ and $\xi_{n}(\theta)$, $n\ge 1$, we introduce the following two modes of uniform convergence:
we write $\xi_{n}(\theta)\cil_{u}\xi_{0}(\theta)$ if
\begin{equation}
|P^{\xi_{n}(\theta)}f - P^{\xi_{0}(\theta)}f |\to_{u}0
\nonumber
\end{equation}
as $n\to\infty$ for every bounded uniformly continuous function $f$, where $P^{\zeta}$ denotes the distribution of $\zeta$;
also, letting $\pr_{\theta}$ denote the distribution of $(\xi_{0}(\theta),\xi_{1}(\theta),\xi_{2}(\theta),\dots)$
we write $\xi_{n}(\theta)\cip_{u}\xi_{0}(\theta)$ if for every $\ep>0$,
\begin{equation}
\pr_{\theta}\{ |\xi_{n}(\theta)-\xi_{0}(\theta)|>\ep\}\to_{u} 0
\nonumber
\end{equation}
as $n\to\infty$.
We will omit the subscript ``$u$'' to denote the convergences without the uniformity. 
We write $a_{n}\lesssim b_{n}$ when there exists a universal multiplicative constant $C$ such that $a_{n}\le C b_{n}$ for every $n$ large enough.
For positive functions $a_{n}(\theta)$ and $b_{n}(\theta)$, we denote $a_{n}(\theta) \lesssim_{u} b_{n}(\theta)$ if $\sup_{\theta\in K}|a_{n}(\theta)/b_{n}(\theta)| \lesssim 1$ for any compact $K\subset(0,2)\times(0,\infty)\times\mbbr$.

\medskip

Let $\D_{j}X:=X_{t^n_{j}}-X_{t^n_{j-1}}$, the $j$th increments of the process $(X_{t})_{t\ge 0}$.
Since $(J_t)_{t\ge0}$ has independent and stationary increments, the log-likelihood function based on observing  $(X_{t^n_j})_{j=0}^{n}$ is given by
\begin{equation}
\ell_n(\theta) :=  \sumj \bigg( \frac{1}{\beta}\log(1/h_n) -  \log \sigma + \log \phi_\beta \left( \ep^n_{j}(\theta)\right) \bigg),
\label{hm:log-lf}
\end{equation}
where $\phi_{\beta}$ denotes the density of $\mathcal{L}(J_{1})$ and
\begin{equation}
\ep^n_{j}(\theta) := \frac{\D_{j}X - h_n\mu}{h_n^{1/\beta}\sig}.
\nn
\end{equation}
The distribution of $(X_t)_{t\ge 0}$ associated with $\theta\in(0,2)\times(0,\infty)\times\mbbr$ is denoted by $\pr_{\theta}$ and the associated expectation operator by $\E_{\theta}$ .
For each $n$ and $\theta$, $\ep^n_{j}(\theta)$'s are $\mcl(J_{1})$-i.i.d.  $\beta$-stable  random variables under $\pr_{\theta}$.

Let  $\vp_{n}:\,(0,2)\times(0,\infty)\times\mbbr \to \mbbr^{3}\otimes\mbbr^{3}$ be a continuous mapping of the block-diagonal form
\begin{equation}
\vp_{n}(\theta) := \frac{1}{\sqrt{n}}
\begin{pmatrix}
\vp_{11,n}(\theta) & \vp_{12,n}(\theta) & 0 \\
\vp_{21,n}(\theta) & \vp_{22,n}(\theta) & 0 \\
0 & 0 &  h_n^{-(1-1/\beta)}\\
\end{pmatrix}.
\label{hm:vp_def}
\end{equation}
This matrix will play the role of the proper rate matrix of the MLE.
As with \cite{BroFuk16}, concerning the upper-left part of $\vp_{n}$ we assume the following conditions:
\begin{equation}
\left\{
\begin{array}{l}
\vp_{11,n}(\theta) \to_{u} \overline{\vp}_{11}(\theta), \\[1mm]
\vp_{12,n}(\theta) \to_{u} \overline{\vp}_{12}(\theta), \\[1mm]
s_{21,n}(\theta):=\beta^{-2}\log(1/h_n)\vp_{11,n}(\theta) + \sig^{-1}\vp_{21,n}(\theta) \to_{u} \overline{\vp}_{21}(\theta), \\[1mm]
s_{22,n}(\theta):=\beta^{-2}\log(1/h_n)\vp_{12,n}(\theta) + \sig^{-1}\vp_{22,n}(\theta) \to_{u} \overline{\vp}_{22}(\theta), \\[2mm]
\ds{\inf_{\theta\in K}|\overline{\vp}_{11}(\theta)\overline{\vp}_{22}(\theta) - \overline{\vp}_{12}(\theta)\overline{\vp}_{21}(\theta)|>0,}
\end{array}
\right.
\label{hm:vp-conditions}
\end{equation}
where the last condition must hold for any compact set $K \subset(0,2)\times(0,\infty)\times\mbbr$;
these conditions naturally come from the form of the normalized score function (see \eqref{hm:add.eq5} below).
For later reference, we note that for every such $K$,
\begin{align}
& \inf_{\theta\in K}|\vp_{11,n}(\theta)\vp_{22,n}(\theta) - \vp_{12,n}(\theta)\vp_{21,n}(\theta)|  \nn\\
& \qquad = \inf_{\theta\in K} \big| \vp_{11,n}(\theta) \sig\{s_{22,n}(\theta)-\beta^{-2}\log(1/h_n)\vp_{12,n}(\theta)\} \nn\\
&{}\qquad\qquad 
-\vp_{12,n}(\theta) \sig \{s_{21,n}(\theta)-\beta^{-2}\log(1/h_n)\vp_{11,n}(\theta)\} \big| \nn\\
& \qquad \gtrsim \inf_{\theta\in K}|\overline{\vp}_{11}(\theta)\overline{\vp}_{22}(\theta) - \overline{\vp}_{12}(\theta)\overline{\vp}_{21}(\theta)| + o(1) \gtrsim 1.
\nonumber
\end{align}
Under \eqref{hm:liminf.p}, it holds that $\sqrt{n}h_n^{1-1/\beta}\to_{u}\infty$ and
\begin{align}
|\vp_{n}(\theta)| &\lesssim_{u} \max\bigg( \frac{\log(1/h_n)}{\sqrt{n}}, \, \frac{1}{\sqrt{n}h_n^{1-1/\beta}} \bigg) \nn\\
&\lesssim_{u} \max\bigg( \frac{\log n}{\sqrt{n}}, \, n^{1/2-1/\beta} \bigg) \to_{u} 0.
\nonumber
\end{align}
Hereafter, we will often omit the dependence on $\theta$ and/or on $n$ from the notations, such as $\vp_{kl,n}$, $\ep_{j}$, $t_j$, $h$ and so on.

\medskip

We deduce from \cite{{DuM73}} (see also the references therein and \cite{MatTak06}) that for each nonnegative integers $k$ and $k'$,
\begin{equation}
|\p^{k}\p_{\beta}^{k'}\phi_{\beta}(y)| \lesssim_{u} (\log|y|)^{k'}|y|^{-\beta-1-k},\quad |y|\to\infty.
\label{hm:dens.ab}
\end{equation}
Let
\begin{equation}
f_{\beta}(y):=\frac{\p_{\beta}\phi_{\beta}}{\phi_{\beta}}(y), \qquad g_{\beta}(y):=\frac{\p\phi_{\beta}}{\phi_{\beta}}(y).
\nonumber
\end{equation}
By the property \eqref{hm:dens.ab} together with the facts that $f_{\beta}$ and $g_{\beta}$ are even and odd, respectively, the symmetric matrix
\begin{equation}
\Sigma(\beta) = \left\{\Sigma_{kl}(\beta)\right\} :=
\begin{pmatrix}
\E_{\theta}\{f_{\beta}(\ep_1)^2\} &  & \text{sym.} \\
\E_{\theta}\{\ep_{1}f_{\beta}(\ep_1)g_{\beta}(\ep_{1})\} & \E_{\theta}\{(1+\ep_{1}g_{\beta}(\ep_{1}))^{2}\} & \\
0 & 0 & \E_{\theta}\{g_{\beta}(\ep_1)^2\} \\
\end{pmatrix},
\label{hm:Sigma_def}
\end{equation}
which only depends on $\beta$ among the components of $\theta$, is continuous in $\beta$, and satisfies that $|\Sigma(\beta)|\lesssim_{u}1$.
Moreover, we introduce the non-degenerate block-diagonal matrix ($\top$ denotes the transpose)
\begin{equation}
\mci(\theta):=\mathrm{diag}\bigg\{
\begin{pmatrix}
\overline{\vp}_{11} & \overline{\vp}_{12} \\
-\overline{\vp}_{21} & -\overline{\vp}_{22}
\end{pmatrix}^{\top}
\begin{pmatrix}
\Sigma_{11} & \Sigma_{12} \\
\Sigma_{21} & \Sigma_{22}
\end{pmatrix}
\begin{pmatrix}
\overline{\vp}_{11} & \overline{\vp}_{12} \\
-\overline{\vp}_{21} & -\overline{\vp}_{22}
\end{pmatrix}
,\, \sig^{-2}\Sigma_{33}
\bigg\},
\nn
\end{equation}
which will turn out to be the asymptotic non-degenerate Fisher information matrix.

The normalized score (central sequence) is defined by
\begin{equation}
\D_{n}(\theta):=\vp_{n}(\theta)^{\top}\p_{\theta}\ell_{n}(\theta),
\nonumber
\end{equation}
and the normalized observed information matrix by
\begin{equation}
\mci_{n}(\theta)=\{\mci_{kl,n}(\theta)\}_{k,l} := - \vp_{n}(\theta)^{\top}\p_{\theta}^{2}\ell_{n}(\theta)\vp_{n}(\theta).
\nn
\end{equation}
The main claim of this section follows below.

\begin{thm}
\begin{enumerate}

\item The uniform LAN property holds:
\begin{equation}
\sup_{u\in K}\bigg|
\ell_{n}\left(\theta+\vp_{n}(\theta)u\right) - \ell_{n}\left(\theta\right) 
- \bigg( u^{\top}\D_{n}(\theta) - \frac{1}{2} u^{\top}\mci(\theta)u \bigg)\bigg| \cip_{u} 0
\nn
\end{equation}
for any compact set $K\subset\mbbr^{3}$, where $\D_{n}(\theta) \cil_{u} N_{3}(0,\mci(\theta))$ and $\mci(\theta)$ is positive definite for each $\theta\in(0,2)\times(0,\infty)\times\mbbr$.

\item There exists a local maximum $\tes$ of $\ell_{n}$ with probability tending to $1$, for which
\begin{equation}
\vp_{n}(\theta)^{-1}(\tes -\theta ) \cil_{u} N_{3}\left(0,\, \mci(\theta)^{-1} \right).
\nonumber
\end{equation}

\end{enumerate}

\label{hm:mle_thm}
 \end{thm}

 \medskip
 
Several remarks on Theorem \ref{hm:mle_thm} are in order.

\begin{enumerate}
\item  The non-degeneracy of $\varphi_n(\theta)$ and $\mci(\theta)$ in Theorem \ref{hm:mle_thm} is essential.
Let $w: \mbbr^{3}\to\mbbrp$ be any non-constant function such that:
$w(x)=w(-x)$; the set $\{x\in\mbbr^{3}:\,w(x)\le c\}$ is convex for each $c>0$; $w$ is continuous at $0$;
finally, $\lim_{|z|\to \infty} e^{-\epsilon |z|^2}w(z) = 0$ for all $\epsilon > 0$.
Then, it follows from the H\'{a}jek-Le Cam asymptotic minimax theorem (\cite{Haj72}, \cite[Chapter II.12]{IbrHas81}, \cite{LeC72}) that 
for each $\tz\in(0,2)\times(0,\infty)\times\mbbr$ and $\del>0$ and for any sequence of estimators $\hat{\theta}_n$ we have
\begin{align}
& \liminf_{n\to \infty} \sup_{\theta:\, |\theta- \theta_0| < \del} 
\E_\theta\left\{ w\left( \varphi_n(\theta_0)^{-1}(\hat{\theta}_n - \theta)\right) \right\}
\nn\\
& {}\qquad \ge
\int_{\mathbb{R}^d}w\left(\mci(\theta_0)^{-1/2}z\right)\phi(z)\mathrm{d}z,
\nonumber
\end{align}
where $\phi$ denotes the density of the three-dimensional standard normal distribution.
Theorem \ref{hm:mle_thm} ensures that a sequence of MLE is asymptotically efficient.
It should be emphasized that the element $\mci(\theta)$ does depend on a spcific choice of the non-diagonal norming matrix $\vp_{n}(\theta)$.

\item Several examples of rate matrices $\vp_{n}=\varphi_n(\theta)$ satisfying the conditions \eqref{hm:vp-conditions} can be elicited. For instance,  the two following rate matrices
	  \begin{equation}\label{phibeta}
	      \varphi_n = \frac{1}{\sqrt{n}}\begin{pmatrix}

					    1			  & 0 &
					    0 \\  -\beta^{-2} \sigma \log (1/h)  & 1 &0 \\
					    0& 0 &  h^{-(1-1/\beta)}
	      
			  \end{pmatrix},
\end{equation}
which gives $\overline{\vp}_{11} = 1$, $\overline{\vp}_{12}=0$, $\overline{\vp}_{21}=0$ and
 $\overline{\vp}_{22} = \sigma^{-1}$ and
 \begin{equation}\label{phis}
	      \varphi_n = \frac{1}{\sqrt{n}}\begin{pmatrix}

					   -1/\log(1/h)		  & 
					    1  & 0 \\ 0  &  -\beta^{-2}\sigma \log (1/h) &0 \\
					    0 & 0 & h^{-(1-1/\beta)}
	      
			  \end{pmatrix},
\end{equation}
which gives $\overline{\vp}_{11} = 0$, $\overline{\vp}_{12}=1$, $\overline{\vp}_{21}=\beta^{-2}$ and
	   $\overline{\vp}_{22} = 0$ can be considered. Remark that  these examples are non-diagonal rate
matrices depending on the scale parameter $\sigma$ and the stable index $\beta$.
For each $\tz\in(0,2)\times(0,\infty)\times\mbbr$ and $\del>0$, using the rate matrix \eqref{phibeta} and the function $w(x,y,z)=x^2$ we can deduce the lower bound
\begin{align}
\liminf_{n\to \infty} \sup_{\theta:\, |\theta- \theta_0| <\del}
\E_\theta\Big\{ \Big(\sqrt{n} (\hat{\beta}_n - \beta)\Big)^2 \Big\} 
\ge \frac{\Sigma_{22}(\beta_{0})}{\Sigma_{11}(\beta_{0})\Sigma_{22}(\beta_{0})-\Sigma_{12}^{2}(\beta_{0})}
\label{hm:lb-beta}
\end{align}
for $\Sigma(\beta_{0})=\{\Sigma_{kl}(\beta_{0})\}$ specified in \eqref{hm:Sigma_def}. 
Likewise, using the rate matrix \eqref{phis} and the function $w(x,y,z)=y^2$, we can deduce that
\begin{align}
& \liminf_{n\to \infty} \sup_{\theta:\, |\theta- \theta_0|<\del} 
\E_\theta\left\{ \left(\frac{\sqrt{n}}{\beta^{-2} \sigma \log(1/h)} (\hat{\sigma}_n - \sigma) \right)^2 \right\} \nn\\
&\qquad \ge \frac{\Sigma_{22}(\beta_{0})}{\Sigma_{11}(\beta_{0})\Sigma_{22}(\beta_{0})-\Sigma_{12}^{2}(\beta_{0})}.
\label{hm:lb-sigma}
\end{align}

\item The explicit form of $\D_{n}(\theta)$ itself will not explicitly appear in the proof of Theorem \ref{hm:mle_thm}.
It is expressed as
\begin{align}
\D_{n}(\theta) = 
\tcr{
\mathrm{diag}\left\{\begin{pmatrix}
\vp_{11,n} & -s_{21,n} \\
\vp_{12,n} & -s_{22,n}
\end{pmatrix}
,~1\right\}
}
\,\frac{1}{\sqrt{n}}\sumj \left(
\begin{array}{c}
f_{\beta}(\ep_{j}) \\ 1+\ep_{j}g_{\beta}(\ep_{j}) \\ g_{\beta}(\ep_{j})
\end{array}
\right),
\label{hm:add.eq5}
\end{align}
for which a direct application of the Lindeberg-Feller theorem yields that $\D_{n}(\theta) \cil N_{3}(0,\mci(\theta))$ under $\pr_{\theta}$.

\item 
What is essential in the derivation of the likelihood asymptotics is the self-similarity of $X$, hence the stable-distribution property of $\D_{j}X$ since the stable {\lp} (including a Wiener process, of course) is the only self-similar one among the whole family of {\lp es}.
Further, in the nature of high-frequency asymptotics, small-time behavior of $\D_{j}X$ is of primary concern.
In this respect, it is readily expected that a similar LAN and/or LAMN phenomena can be deduced for locally stable {\lp es}, 
and even more generally, for non-linear stochastic differential equations driven by a locally stable {\lp}, 
where a {\lp} $J'$ is said to be locally $\beta$-stable if $\mcl(h^{-1/\beta}J'_{h})$ weakly convergence to $\mcl(J_{1})$ as $h\to 0$.
Of course, such considerations must require expertise of stochastic calculus on Poisson space.
We refer to \cite{IvaKulMas15} and \cite{CleGlo15} for related previous studies in case of known $\beta$.
\end{enumerate}

\medskip

\begin{proof}[Proof of Theorem \ref{hm:mle_thm}]
We will complete the proof by verifying Sweeting's conditions [C1] and [C2] in \cite{Swe80}, which here read as follows.
\begin{itemize}
\item Uniform convergence of the observed information matrix:
\begin{equation}
\mci_{n}(\theta) \cip_{u} \mci(\theta),
\label{hm:sw-1}
\end{equation}
where $\mci(\theta)$ is positive definite for each $\theta$.

\item Uniform growth rate of the norming matrix $\vp_{n}$:
\begin{equation}
\sup_{\theta' \in \mathfrak{N}_{n}(c;\theta)} | \vp_{n}(\theta')^{-1}\vp_{n}(\theta) - I_{3} | \to_{u} 0
\label{hm:C2i-1}
\end{equation}
for any $c>0$, where
\begin{equation}
\mathfrak{N}_{n}(c;\theta) := \left\{\theta' \in (0,2)\times(0,\infty)\times\mbbr:\, |\vp_{n}(\theta)^{-1}(\theta'-\theta)|\le c \right\},
\nonumber
\end{equation}
a shrinking neighborhood of $\theta$, and where $I_{k}$ denotes the $k\times k$-identity matrix.

\item Yet another uniform convergence of $\mci_{n}(\theta)$:
\begin{equation}
\sup_{\theta^{1},\theta^{2},\theta^{3} \in \mathfrak{N}_{n}(c;\theta)}
\left| \vp_{n}(\theta)^{\top}\{ \p_{\theta}^{2}\ell_{n}(\theta^{1},\theta^{2},\theta^{3})-\p_{\theta}^{2}\ell_{n}(\theta) \} \vp_{n}(\theta) \right| \to_{u} 0
\label{hm:C2ii-1}
\end{equation}
for each $c>0$, where
\begin{equation}
\p_{\theta}^{2}\ell_{n}(\theta^{1},\theta^{2},\theta^{3}) := 
\begin{pmatrix}
\p_{\beta}^{2}\ell_{n}(\theta^{1}) & \p_{\beta}\p_{\sig}\ell_{n}(\theta^{1}) & \p_{\beta}\p_{\mu}\ell_{n}(\theta^{1}) \\
\p_{\beta}\p_{\sig}\ell_{n}(\theta^{2}) & \p_{\sig}^{2}\ell_{n}(\theta^{2}) & \p_{\sig}\p_{\mu}\ell_{n}(\theta^{2}) \\
\p_{\beta}\p_{\mu}\ell_{n}(\theta^{3}) & \p_{\sig}\p_{\mu}\ell_{n}(\theta^{3}) & \p_{\mu}^{2}\ell_{n}(\theta^{3}) \\
\end{pmatrix}.
\nonumber
\end{equation}

\end{itemize}
Having proved these three claims, by means of \cite[Theorems 1 and 2]{Swe80} it is straightforward to deduce the assertions as in \cite[Theorem 2.10]{Mas15LM}; although the latter theorem deals with the diagonal norming, the proof can be traced without any essential change.

\paragraph{Proof of \eqref{hm:sw-1}.}

Since $\E_{\theta}\{f_{\beta}(\ep_1)\}=0$, we have $\Sigma_{12}=\E_{\theta}\{f_{\beta}(\ep_1)(1+ \ep_{1}g_{\beta}(\ep_{1}))\}$.
Schwarz's inequality then shows that $\Sigma_{12}^{2}<\Sigma_{11}\Sigma_{22}$, hence we also have $|\Sigma(\beta)|\gtrsim_{u}1$:
$\Sigma(\beta)$ is positive definite for each $\beta\in(0,2)$.
It remains to show that
\begin{equation}
\left| \E_{\theta}\{\mci_{kl,n}(\theta)\} - \mci_{kl}(\theta) \right| + \var_{\theta}\{\mci_{kl,n}(\theta)\} \to_{u} 0,\qquad k,l\in\{1,2,3\};
\nonumber
\end{equation}
we refer to \cite[the proof of Theorem 3.2]{Mas15LM} for the explicit expressions of all the elements of $\p_{\theta}^{2}\ell_{n}(\theta)$.
We will repeatedly make use of the following basic fact without mention:
for a row-wise independent triangular array $\{\xi_{nj}(\theta)\}$ with $\E_{\theta}\{\xi_{nj}(\theta)\}=0$ for each $\theta$, we have $\sumj \xi_{nj}(\theta) \cip_{u} c(\theta)$
if both $\sumj \E_{\theta}\{\xi_{nj}(\theta)\} \to_{u} c(\theta)$ and $\sumj \E_{\theta}\{\xi_{nj}^{2}(\theta)\} \to_{u} 0$.

\medskip

Since $\E_{\theta}\{(\p^{2}\phi_{\beta}/\phi_{\beta})(\ep_{j})\}=\int \{\p^{2}\phi_{\beta}(y)\}dy = 0$ for each $\theta$, we have
\begin{align}
\E_{\theta}\{\mci_{33,n}(\theta)\} 
&=\E_{\theta}\bigg( \frac{1}{n}\sumj\sig^{-2}\left\{g_{\beta}(\ep_{j})^{2}-(\p^{2}\phi_{\beta}/\phi_{\beta})(\ep_{j})\right\} \bigg) = \mci_{33}(\theta).
\nonumber
\end{align}
By \eqref{hm:dens.ab} we have $\var_{\theta}\{\mci_{33,n}(\theta)\}\lesssim_{u}1/n$, 
hence $\mci_{33,n}(\theta) \cip_{u} \mci_{33}(\theta)$. We can handle the lower-left $1\times 2$-part of $\mci_{n}(\theta)$ in a similar manner: we have
\begin{align}
\E_{\theta}\{\mci_{31,n}(\theta)\} = -\frac{\vp_{21,n}}{\sig}\E\left( g_{\beta}(\ep_{1})+\ep_{1}\p g_{\beta}(\ep_{1})\right)
+\frac{\vp_{11,n}}{\sig}\E\{\p_{\beta}g_{\beta}(\ep_{1})\} =0 =\mci_{31}(\theta)
\nonumber
\end{align}
and
\begin{align}
\E_{\theta}\{\mci_{32,n}(\theta)\} = -\frac{\vp_{22,n}}{\sig}\E\left( g_{\beta}(\ep_{1})+\ep_{1}\p g_{\beta}(\ep_{1})\right) = 0=\mci_{32}(\theta),
\nonumber
\end{align}
and again \eqref{hm:dens.ab} controls their variances, concluding that $\mci_{31,n}(\theta) \cip_{u} \mci_{31}(\theta)$ and $\mci_{32,n}(\theta) \cip_{u} \mci_{32}(\theta)$.

Turning to the upper-left $2\times 2$-part of $\mci_{n}(\theta)$, we first note that the uniform convergences in probability:
\begin{align}
Q_{1,n} &:= \frac{1}{n}\sumj f_{\beta}(\ep_{j})^{2} \cip_{u} \Sigma_{11}, \nn\\
Q_{2,n} &:= -\frac{1}{n}\sumj \ep_{j}\p_{\beta}g_{\beta}(\ep_{j}) = -\frac{1}{n}\sumj \ep_{j}\frac{\p_{\beta}\p\phi_{\beta}}{\phi_{\beta}}(\ep_{j}) + \frac{1}{n}\sumj \ep_{j}f_{\beta}(\ep_{j})g_{\beta}(\ep_{j}) \nn\\
&\cip_{u} \Sigma_{12}, \nn\\
Q_{3,n} &:= \frac{1}{n}\sumj \{-\ep_{j}g_{\beta}(\ep_{j}) -\ep_{j}^{2}\p g_{\beta}(\ep_{j})\} \nn\\
&= \frac{1}{n}\sumj \left[ \left\{1+\ep_{j}g_{\beta}(\ep_{j})\right\}^{2} - \left\{1+3\ep_{j}g_{\beta}(\ep_{j}) + \ep_{j}^{2} (\p^{2}\phi_{\beta}/\phi_{\beta})(\ep_{j}) \right\} \right] \nn\\
&\cip_{u} \Sigma_{22}.
\nonumber
\end{align}
Under \eqref{hm:vp-conditions}, we can see that the elements of $\{\mci_{kl,n}(\theta):\, k,l\in\{1,2\}\}$ involving the diverging factors ``$\log(1/h)$'' and ``$\{\log(1/h)\}^{2}$'' are completely canceled out, thereby resulting in \eqref{hm:sw-1}:
\begin{align}
& -\frac{1}{n}
\begin{pmatrix}
\vp_{11,n} & \vp_{12,n} \\
\vp_{21,n} & \vp_{22,n}
\end{pmatrix}^{\top}
\p_{(\beta,\sig)}^{2}\ell_{n}(\theta)
\begin{pmatrix}
\vp_{11,n} & \vp_{12,n} \\
\vp_{21,n} & \vp_{22,n}
\end{pmatrix} \nn\\
&=
\begin{pmatrix}
\vp_{11,n} & \vp_{12,n} \\
\vp_{21,n} & \vp_{22,n}
\end{pmatrix}^{\top}
\begin{pmatrix}
Q_{1,n}-2\beta^{-2}\log(1/h) Q_{2,n} + \beta^{-4}\log^{2}(1/h) Q_{3,n} & \text{sym.} \\
-\sig^{-1}Q_{2,n}+\sig^{-1}\beta^{-2}\log(1/h)Q_{3,n} & \sig^{-2}Q_{3,n}
\end{pmatrix}
\nn\\
&{}\qquad\qquad \times 
\begin{pmatrix}
\vp_{11,n} & \vp_{12,n} \\
\vp_{21,n} & \vp_{22,n}
\end{pmatrix}
\tcr{
+ O_p\left(\frac{\log(1/h)}{\sqrt{n}}\right)
}
\nn\\
&=\left(
\begin{array}{c}
\vp_{11,n}^{2}Q_{1,n} - 2\vp_{11,n}s_{21,n}Q_{2,n} + s_{21,n}^{2}Q_{3,n} \\
\vp_{11,n}\vp_{12,n}Q_{1,n} - (\vp_{11,n}s_{22,n}+\vp_{12,n}s_{21,n})Q_{2,n}+ s_{21,n}s_{22,n}Q_{3,n}
\end{array}
\right. \nn\\
& \hspace{3cm}\left.
\begin{array}{c}
\text{sym.} \\
\vp_{12,n}^{2}Q_{1,n} - 2\vp_{12,n}s_{22,n}Q_{2,n} + s_{22,n}^{2}Q_{3,n}
\end{array}
\right)
\tcr{
+ O_p\left(\frac{\log(1/h)}{\sqrt{n}}\right)
}
\nn\\
&=
\begin{pmatrix}
\vp_{11,n} & \vp_{12,n} \\
-s_{21,n} & -s_{22,n}
\end{pmatrix}^{\top}
\begin{pmatrix}
Q_{1,n} & Q_{2,n} \\
Q_{2,n} & Q_{3,n}
\end{pmatrix}
\begin{pmatrix}
\vp_{11,n} & \vp_{12,n} \\
-s_{21,n} & -s_{22,n}
\end{pmatrix}
\tcr{
+o_p(1)
}
\nn\\
& \cip_{u}
\begin{pmatrix}
\mci_{11}(\theta) & \mci_{12}(\theta) \\
\mci_{21}(\theta) & \mci_{22}(\theta)
\end{pmatrix}.
\nonumber
\end{align}

\paragraph{Proof of \eqref{hm:C2i-1}.}

Fix $c>0$ in the rest of this proof. Write
\begin{equation}
n^{-1/2}\tilde{\vp}_{n}=(n^{-1/2}\tilde{\vp}_{kl,n})_{k,l\in\{1,2\}}
\nonumber
\end{equation}
for the upper-left $2\times 2$-part of $\vp_{n}$. Then it suffices to show that
\begin{equation}
\sup_{\theta' \in \mathfrak{N}_{n}(c;\theta)} \bigg(| \tilde{\vp}_{n}(\theta')^{-1}\tilde{\vp}_{n}(\theta) - I_{2} |
\vee | h^{1/\beta-1/\beta'} - 1 | \bigg) \to_{u} 0.
\nonumber
\end{equation}
To this end, we first prove
\begin{equation}
\sup_{\theta' \in \mathfrak{N}_{n}(c;\theta)}
|\sqrt{n}(\beta'-\beta)| \vee \bigg|\frac{\sqrt{n}}{\log(1/h)}(\sig'-\sig)\bigg| \lesssim_{u} 1.
\label{hm:C2i-2}
\end{equation}
Note that $| \sqrt{n}\tilde{\vp}_{n}(\theta)^{-1}(\rho'-\rho)| \le c$ for $\theta' \in \mathfrak{N}_{n}(c;\theta)$,
where $\rho:=(\beta,\sig)^{\top}$.
Let
\begin{equation}
\overline{\del}=\overline{\del}(\theta):=\overline{\vp}_{11}\overline{\vp}_{22} - \overline{\vp}_{12}\overline{\vp}_{21},
\nonumber
\end{equation}
whose absolute value under \eqref{hm:vp-conditions} is bounded away from $0$: $|\overline{\del}(\theta)|^{-1} \lesssim_{u} 1$. For convenience, we will write
\begin{equation}
\text{$\al_{n}(\theta,\theta')=o_{u,c}(1)$ (resp. $=O_{u,c}(1)$) \quad if \quad $\sup_{\theta' \in \mathfrak{N}_{n}(c;\theta)}|\al_{n}(\theta,\theta')|\to_{u}0$ (resp. $\lesssim_{u}1$)}
\nonumber
\end{equation}
for any function $\al_{n}$ on $\Theta^{2}$. Straightforward algebra leads to the identity
\begin{equation}
\sqrt{n}\tilde{\vp}_{n}(\theta)^{-1}(\rho'-\rho) = A_{n}(\theta) b_{n}(\theta,\theta')
\label{hm:me-1}
\end{equation}
with $A_{n}(\theta) \in \mbbr^{2}\otimes\mbbr^{2}$ satisfying that
\begin{equation}
A_{n}(\theta) = \{\overline{\del} + o_{u,c}(1)\}^{-1}
\begin{pmatrix}
s_{22,n} & -\vp_{12,n} \\
-s_{21,n} & \vp_{11,n}
\end{pmatrix}
\nonumber
\end{equation}
and
\begin{equation}
b_{n}(\theta,\theta') := 
\begin{pmatrix}
\sig\sqrt{n}(\beta'-\beta) \\
\sig\beta^{-2}\log(1/h) \sqrt{n}(\beta'-\beta) + \sqrt{n}(\sig'-\sig)
\end{pmatrix}.
\nonumber
\end{equation}
By the continuity in $\theta$, we have
\begin{equation}
A_{n}(\theta)^{\top}A_{n}(\theta) = \frac{1}{\overline{\del}^{2}}
\begin{pmatrix}
\overline{\vp}_{22}^{2}+\overline{\vp}_{21}^{2} & \text{sym.} \\
-\overline{\vp}_{12}\overline{\vp}_{22}-\overline{\vp}_{21}\overline{\vp}_{11} & \overline{\vp}_{11}^{2}+\overline{\vp}_{12}^{2}
\end{pmatrix} + o_{u,c}(1).
\nonumber
\end{equation}
The minimum eigenvalue of the first term in the right-hand side is
\begin{equation}
\underline{\lam}_{n}=\underline{\lam}_{n}(\theta) := \frac{1}{2} \overline{\del}^{-2} (f-\sqrt{f^{2}-4g})
\nonumber
\end{equation}
with $f:= \overline{\vp}_{11}^{2}+\overline{\vp}_{22}^{2}+\overline{\vp}_{12}^{2}+\overline{\vp}_{21}^{2}$ and $g:= (\overline{\vp}_{11}\overline{\vp}_{22}-\overline{\vp}_{12}\overline{\vp}_{21})^{2}$; 
by completing squares, we see that the term $f^{2}-4g$ is nonnegative.
Since $f \mapsto f-\sqrt{f^{2}-4g}\,(> 0)$ is monotonically decreasing on $(2\sqrt{g},\infty)$ and $|f|\lesssim_{u}1$, we conclude that
\begin{equation}
|\underline{\lam}_{n}(\theta)|^{-1}\lesssim_{u}1.
\nonumber
\end{equation}
Therefore, by \eqref{hm:me-1} we have
\begin{align}
& \sup_{\rho:\,| \sqrt{n}\tilde{\vp}_{n}(\theta)^{-1}(\rho'-\rho)| \le c}|b_{n}(\theta,\theta')| \nn\\
&= \sup_{\rho:\,| \sqrt{n}\tilde{\vp}_{n}(\theta)^{-1}(\rho'-\rho)| \le c}\left| A_{n}(\theta) ^{-1}\sqrt{n}\tilde{\vp}_{n}(\theta)^{-1}(\rho'-\rho) \right| \nn\\
&\le c|\underline{\lam}_{n}(\theta)|^{-1/2} \lesssim_{u}1,
\nonumber
\end{align}
so that $|\sqrt{n}(\beta'-\beta)| \vee |\sqrt{n}(\log(1/h))^{-1}(\sig'-\sig)| = O_{u,c}(1)$, hence \eqref{hm:C2i-2}.

\medskip

Next, using \eqref{hm:C2i-2}, which in particular implies that $\sup_{\theta'\in\mathfrak{N}_{n}(c;\theta)}\log(1/h)\{|\beta'-\beta| \vee |\sig'-\sig|\} \to_{u}0$, we see that
\begin{equation}
\sup_{\theta'\in\mathfrak{N}_{n}(c;\theta)}|\tilde{\vp}_{n}(\theta')^{-1}\tilde{\vp}_{n}(\theta) - I_{2}|\to_{u}0.
\label{hm:me-2}
\end{equation}
Indeed, the $(2,1)$th entry equals
\begin{align}
& \{ \sig'\overline{\del}(\theta') + o_{u,c}(1)\}^{-1}\{\overline{\vp}_{21}(\theta)\overline{\vp}_{11}(\theta')-\overline{\vp}_{11}(\theta)\overline{\vp}_{21}(\theta')\} \nn\\
&= \sig^{\prime -1}\overline{\del}(\theta')^{-1}\overline{\vp}_{11}(\theta)\overline{\vp}_{11}(\theta') \nn\\
&{}\qquad \times
\left\{ \beta^{\prime -2}\log(1/h)(\sig'-\sig) + \sig \log(1/h)(\beta^{\prime -2}-\beta^{-2}) \right\} + o_{u,c}(1) \nn\\
&= o_{u,c}(1),
\nonumber
\end{align}
and the $(1,2)$, $(1,1)$ and $(2,2)$th entries can be handled similarly.
By means of the mean-value theorem and \eqref{hm:C2i-2},
\begin{equation}
| h^{1/\beta-1/\beta'} - 1 | \lesssim_{u} | h^{1/\beta-1/\beta''} \log(1/h)(\beta'-\beta) |
\lesssim h^{1/\beta-1/\beta''}\cdot o_{u}(1)
\nn
\end{equation}
for a point $\beta''$ on the segment joining $\beta$ and $\beta'$. 
It follows from \eqref{hm:C2i-2} that
\begin{equation}
\sup_{\theta'\in\mathfrak{N}_{n}(c;\theta)} h^{1/\beta-1/\beta''} \le \sup_{\theta'\in\mathfrak{N}_{n}(c;\theta)} (h^{-|\beta''-\beta|})^{1/(\beta''\beta)} \lesssim_{u} 1,
\nn
\end{equation}
followed by $\sup_{\theta' \in \mathfrak{N}_{n}(c;\theta)} | h^{1/\beta-1/\beta'} - 1 | \to_{u} 0$.
This combined with \eqref{hm:me-2} yields \eqref{hm:C2i-1}.

\paragraph{Proof of \eqref{hm:C2ii-1}.}

For $\theta'\in\mathfrak{N}_{n}(c;\theta)$ we have
\begin{equation}
|\ep_{j}(\theta')| = \bigg| \bigg( \ep_{j}(\theta) + \frac{1}{\sig\sqrt{n}}\sqrt{n}h^{1-1/\beta}(\mu-\mu') \bigg)\frac{\sig}{\sig'}h^{1/\beta-1/\beta'} \bigg| \lesssim_{u} |\ep_{j}(\theta)| + o_{u,c}(1).
\nonumber
\end{equation}
It follows from \eqref{hm:log-lf} and \eqref{hm:dens.ab} that for each $k,l,m$,
\begin{equation}
\frac{1}{n}\big|\p_{\beta}^{k}\p_{\sig}^{l}\p_{\mu}^{m}\ell_{n}(\theta)\big| \lesssim_{u} \{\log(1/h)\}^{k}h^{(1-1/\beta)m} \frac{1}{n}\sumj \left\{1 + \log\left(1+|\ep_{j}(\theta)|^{2}\right)\right\}^{k}.
\nn
\end{equation}
Hence, we can (rather roughly) estimate as follows:
\begin{align}
& \sup_{\theta^{1},\theta^{2},\theta^{3} \in \mathfrak{N}_{n}(c;\theta)}
\left| \vp_{n}(\theta)^{\top}\{ \p_{\theta}^{2}\ell_{n}(\theta^{1},\theta^{2},\theta^{3})-\p_{\theta}^{2}\ell_{n}(\theta) \} \vp_{n}(\theta) \right| \nn\\
& \lesssim_{u} 
\sup_{\theta^{1},\theta^{2},\theta^{3},\theta' \in \mathfrak{N}_{n}(c;\theta)}
\left| \vp_{n}(\theta)^{\top} \left\{ \p_{\theta}^{3}\ell_{n}(\theta^{1},\theta^{2},\theta^{3}) [\theta' - \theta] \right\} \vp_{n}(\theta) \right|
\nn\\
&\lesssim_{u} \frac{1}{\sqrt{n}} \{\log(1/h)\}^{6} \sup_{\{\beta',\beta''\}} h^{(1/\beta'-1/\beta'')3} 
\sup_{\theta' \in \mathfrak{N}_{n}(c;\theta)} \frac{1}{n}\sumj \left\{1 + \log\left(1+|\ep_{j}(\theta')|\right)\right\}^{3} \nn\\
&\lesssim_{u} \frac{1}{\sqrt{n}} \{\log(1/h)\}^{6} \frac{1}{n}\sumj \left\{1 + \log\left(1+|\ep_{j}(\theta)|\right)\right\}^{3}
\label{hm:proof+1} \\
&\to_{u}0,
\nn
\end{align}
where the supremum $\sup_{\{\beta',\beta''\}}$ is taken over $\beta',\beta''\in(0,2)$ both lying in a closed ball with center $\beta$ and radius of $c'/\log(1/h)$ for some constant $c'>0$ possibly depending on $c$;
for the estimate \eqref{hm:proof+1}, we made use of the last part of the proof of \eqref{hm:C2i-1} and also the fact that
$\sup_{\theta' \in \mathfrak{N}_{n}(c;\theta)}|\ep_{j}(\theta')| \lesssim_{u} |\ep_{j}(\theta)| + o_{u,c}(1)$
to replace $\ep_{j}(\theta')$ by $\ep_{j}(\theta)$.
\end{proof}

\section{Simple efficient estimator}\label{hm:sec_optest}

The log-likelihood function \eqref{hm:log-lf} is rather complicated and not given in a closed form except for the Cauchy case $\beta=1$.
In particular, by the parameter $\beta$ which is here also present inside the density $\phi_{\beta}$, the optimization issue has difficulty not shared with the classical i.i.d. setting where $h\equiv 1$, hence the standard library for fitting a stable distribution on computer, such as \cite{stable5.3} and those cited in \cite{MatTak06}, cannot be of direct use.
Still, once an easy-to-compute initial estimator having an asymptotic behavior good enough, we can benefit from the classical scoring, which enables us to bypass time-consuming numerical search for $\tes$.
The purpose of this section is to provide such an efficient estimator.

We keep using the notations introduced in Section \ref{hm:sec_lfa}.
Throughout this section, we fix a true value $\theta\in(0,2)\times(0,\infty)\times\mbbr$, and proceed with the single image measure $\pr:=\pr_{\theta}$ (so $\E:=\E_{\theta}$).

\subsection{Scoring procedure and asymptotic normality}
\label{hm:sec_optest-1}

By Theorem \ref{hm:mle_thm} we can consider a sequence of $\sig(X_{t_{j}}:\, j\le n)$-measurable random variables $\tes$ such that $\pr\{ \p_{\theta}\ell_{n}(\tes)=0 \} \to 0$ and that $\vp_{n}(\theta)^{-1}(\tes-\theta) \cil N_{3}(0,\mci(\theta)^{-1})$.
Suppose for a moment that we have an estimator $\tes^0$ such that
\begin{equation}
\vp_{n}(\theta)^{-1}(\tes^0 -\theta) = O_{p}(1).
\label{hm:theta0}
\end{equation}
Define the one-step MLE $\tes^1$ by
\begin{equation}
\tes^1 = \tes^0 + \left( \vp_{n}(\tes^0)^{-1 \top} \mci(\tes^0) \vp_{n}(\tes^0)^{-1} \right)^{-1} \p_{\theta}\ell_{n}(\tes^0),
\label{hm:1step.est}
\end{equation}
which is well-defined with probability tending to $1$.
Then, by means of Fisher's scoring together with making use of the properties \eqref{hm:sw-1}, \eqref{hm:C2i-1}, and \eqref{hm:C2ii-1}, we see that $\tes^1$ is asymptotically equivalent to the MLE, that is $\vp_{n}(\theta)^{-1}(\tes-\tes^1) \cip 0$.
By Theorem \ref{hm:mle_thm}, this implies that
\begin{equation}
\vp_{n}(\theta)^{-1}(\tes^1 - \theta) \cil N_{3}\left( 0,\, \mci(\theta)^{-1}\right).
\label{hm:asymp.efficiency}
\end{equation}
We remark that the likelihood function $\ell_{n}(\theta)$, its partial derivatives and the components of $\Sigma(\beta)$ (recall \eqref{hm:Sigma_def}) can be computed through numerical integration with high precision (see \cite{MatTak06} and the references therein).

\subsection{Initial estimator}
\label{hm:sec_optest-2}

We here show that the method of moments based on appropriate moment matchings with sample median adjustment provides us with an initial estimator $\tes^0$ satisfying \eqref{hm:theta0}. We will follow the same scenario as in \cite{Mas09jjss} (also \cite[Section 3.3]{Mas15LM}), as described below.

For brevity (yet without loss of generality) we set $n$ to be odd, say $n=2k+1$. Let
\begin{equation}
\hat{\mu}^{0}_{n} := \frac{1}{h}\D_{(k+1)}X,
\label{hm:smed_mu0}
\end{equation}
where $\D_{(1)}X \le \dots \le \D_{(n)}X$ denote the order statistics of $(\D_{j}X)_{j=1}^{n}$.
From theory of order statistics or that of the least absolute deviation estimation, we know that $\hat{\mu}^{0}_{n}$ is rate-efficient (\cite{Mas09jjss}, \cite{Mas15LM}):
\begin{equation}
\sqrt{n}h^{1-1/\beta}(\hat{\mu}^{0}_{n}-\mu) 
= -\frac{\sig}{2\phi_{\beta}(0)}\frac{1}{\sqrt{n}}\sumj\mathrm{sgn}\left(\ep_{j}\right) + o_{p}(1),
\label{hm:lad.se}
\end{equation}
its asymptotic distribution being $N\left(0,\, \sig^{2}/\{4\phi_{\beta}(0)^{2}\}\right)$.

Let $g=(g_{1},g_{2}): \mbbr\setminus\{0\} \to\mbbr^{2}$ be a function such that the following two conditions hold:
there exist a constant $\ep_{0}>0$ and a Lebesgue integrable function $\zeta$ such that
\begin{equation}
\sup_{|a|\le\ep_{0}}\left|\p \phi_{\beta}(y+a)\right|\le\frac{\zeta(y)}{1+|g(y)|^{3}},\qquad y\in\mbbr\setminus\{0\},
\label{hm:g_A1}
\end{equation}
and moreover
\begin{equation}
\int_{\mbbr\setminus\{0\}}|g(y)|^{3}\phi_{\beta}(y)dy<\infty.
\label{hm:g_A2}
\end{equation}
Then we consider
\begin{align}
\mbbg_{n}(\beta') &:= \frac{1}{2k}\sum_{j\ne k+1} g\left(h_{n}^{-1/\beta'}(\D_{(j)}X-\hat{\mu}^{0}_{n} h)\right), \nn\\
\mbbg_{0}(\beta',\sig') &
:= \int_{\mbbr\setminus\{0\}} g(y) \sig^{\prime -1}\phi_{\beta'}(\sig^{\prime -1}y)dy,
\nonumber
\end{align}
where we removed the index $j=k+1$ for the summation to subsume cases where $g(0)$ is not defined, such as the case $y\mapsto\log|y|$ considered later on.
Note that at this stage the elements of the vector $\mbbg_{0}(\beta',\sig')$ may be infinite, or even not well-defined, depending on specific value of $(\beta',\sig')$.
It will be convenient to introduce an auxiliary value $\beta_{g}\in[0,2)$ such that
\begin{equation}
\beta \in (\beta_{g} , 2)
\label{hm:add.eq2}
\end{equation}
(recall that we are fixing a true value $\theta=(\beta,\sig,\mu)$); this form is enough to cover the example $g_{q}$ in Section \ref{sec:sim},
where we will demonstrate that the restriction \eqref{hm:add.eq2} does not force us to narrow the parameter space of $\beta$ beforehand.
Under the conditions \eqref{hm:g_A1} and \eqref{hm:g_A2} we can apply the median-adjusted central limit theorem \cite[Theorem 3.7]{Mas15LM} for $\sqrt{n}(\mbbg_{n}(\beta) - \mbbg_{0}(\beta,\sig))$, hence in particular $\mbbg_{n}(\beta) \cip \mbbg_{0}(\beta,\sig)$.
Assume that the random equation
\begin{equation}
\mbbg_{n}(\beta') = \mbbg_{0}(\beta',\sig')
\label{hm:ME_random.eq}
\end{equation}
with respect to $(\beta',\sig') \in (\beta_{g},2)\times(0,\infty)$ has a unique solution $(\bes^0,\ses^0)$ with $\pr$-probability tending to $1$, and let
\begin{equation}
\tes^{0}:=(\bes^0,\ses^0,\hat{\mu}^{0}_{n}).
\nonumber
\end{equation}

Now, we claim that \eqref{hm:theta0} holds for any norming matrix $\vp_{n}(\theta)$ satisfying \eqref{hm:vp-conditions} (though we here do not require the uniformity in $\theta$), if the following conditions hold:
\begin{align}
& \text{$\mbbg_{n} \in \mcc^{1}\left((0,2)\right)$ a.s. 
and $\mbbg_{0} \in \mcc^{2}\left( (\beta_{g},2) \times(0,\infty)\right)$;}
\label{hm:me_A3}\\
& \left|\sqrt{n}(\bes^0 -\beta)\right| + \bigg|\frac{\sqrt{n}}{\log(1/h)}(\ses^0 -\sig)\bigg| = O_{p}(1);
\label{hm:me_A1}\\
& -\p_{\beta}\mbbg_{n}(\tilde{\beta}_{n}) = \frac{\sig}{\beta^{2}}\log(1/h) \p_{\sig}\mbbg_{0}(\beta,\sig) + o_{p}\bigg(\frac{1}{\log(1/h)}\bigg);
\label{hm:me_A2}\\
& \text{$\Gam_{0}(\beta,\sig):=\ds{\bigl( \p_{\beta}\mbbg_{0}(\beta,\sig),\, -\sig\p_{\sig}\mbbg_{0}(\beta,\sig) \bigr)}$ is non-singular},
\label{hm:me_A4}
\end{align}
where \eqref{hm:me_A2}, which comes from the rate-matrix condition \eqref{hm:vp-conditions}, should holds for any $\tilde{\beta}_{n}$ such that $\sqrt{n}(\tilde{\beta}_{n}-\beta)=O_{p}(1)$.
As a matter of fact, it will turn out that the distribution of $\vp_{n}(\theta)^{-1}(\tes^0 -\theta)$ is asymptotically (non-degenerate) normally distributed.

\medskip

Recall the notation $\rho=(\beta,\sig)^{\top}$ and that $\tilde{\vp}_{n}(\rho)$ denotes the upper-left $2\times 2$-submatrix of $\sqrt{n}\vp_{n}(\theta)$ (see the expression \eqref{hm:vp_def}). Write $\res^0=(\bes^0,\ses^0)^{\top}$ and $\mbbs_{n}(\rho)=\sqrt{n}( \mbbg_{n}(\beta) - \mbbg_{0}(\beta,\sig) )$.
To deduce the claim, we expand $\mbbs_{n}(\res^0)$ around $\rho$ on the event $\{\mbbs_{n}(\res^0)=0\}$:
\begin{align}
\bigg( -\frac{1}{\sqrt{n}}\p_{\rho}\mbbs_{n}(\tilde{\rho}_{n}) \tilde{\vp}_{n}(\rho) \bigg) \sqrt{n}\tilde{\vp}_{n}(\rho)^{-1}(\res^0-\rho) = \mbbs_{n}(\rho),
\nonumber
\end{align}
where $\tilde{\rho}_{n}=(\tilde{\beta}_{n},\tilde{\sig}_{n})^{\top}$ is a point on the segment joining $\res^0$ and $\rho$.
It follows from \eqref{hm:me_A3} and \eqref{hm:me_A1} that $\tilde{\beta}_{n}$ is $\sqrt{n}$-consistent and that
\begin{equation}
\p_{\rho}\mbbg_{0}(\tilde{\rho}_{n})=\p_{\rho}\mbbg_{0}(\rho)+O_{p}\bigg(\frac{\log(1/h)}{\sqrt{n}}\bigg)
=\p_{\rho}\mbbg_{0}(\rho) + o_{p}\bigg(\frac{1}{\log(1/h)}\bigg).
\nonumber
\end{equation}
Further, by \eqref{hm:vp-conditions} we have $|\tilde{\vp}_{n}(\rho)|=O(\log(1/h))$.
Building on these observations combined with \eqref{hm:me_A2} and direct computations based on \eqref{hm:vp-conditions}, we obtain the following chain of equalities between $2\times 2$ matrices:
\begin{align}
-\frac{1}{\sqrt{n}}\p_{\rho}\mbbs_{n}(\tilde{\rho}_{n}) \tilde{\vp}_{n}(\rho)
&= \bigl( -\p_{\beta}\mbbg_{n}(\tilde{\beta}_{n}) + \p_{\beta}\mbbg_{0}(\tilde{\rho}_{n}),
~\p_{\sig}\mbbg_{0}(\tilde{\rho}_{n}) \bigr) \tilde{\vp}_{n}(\rho) \nn\\
&= \biggl( \frac{\sig}{\beta^{2}}\log(1/h) \p_{\sig}\mbbg_{0}(\rho)  + \p_{\beta}\mbbg_{0}(\rho),~ \p_{\sig}\mbbg_{0}(\rho) \biggr) \tilde{\vp}_{n}(\rho) + o_{p}(1) \nn\\
&= \Gam_{0}(\rho)
\begin{pmatrix}
\overline{\vp}_{11} & \overline{\vp}_{12} \\
-\overline{\vp}_{21} & -\overline{\vp}_{22}
\end{pmatrix} + o_{p}(1).
\nonumber
\end{align}
This clearly shows that the condition \eqref{hm:me_A2}, which may seem unusual at first glance, combined with \eqref{hm:vp-conditions} is quite natural in our framework.
By \cite[Theorem 3.7]{Mas15LM}, the random variable $\mbbs_{n}(\rho)$ is asymptotically normally distributed with non-degenerate asymptotic covariance matrix.
Piecing together what we have seen and the invertibility condition \eqref{hm:me_A4} gives
\begin{align}
\sqrt{n}\tilde{\vp}_{n}(\rho)^{-1}(\res^0 - \rho) 
=\begin{pmatrix}
\overline{\vp}_{11} & \overline{\vp}_{12} \\
-\overline{\vp}_{21} & -\overline{\vp}_{22}
\end{pmatrix}^{-1}
\Gam_{0}(\rho)^{-1}\mbbs_{n}(\rho)+o_{p}(1).
\label{hm:rho.se}
\end{align}
The tightness \eqref{hm:theta0} now follows from the stochastic expansions \eqref{hm:lad.se} and \eqref{hm:rho.se} together with the form \eqref{hm:vp_def} of $\vp_{n}(\theta)$.

\subsection{Asymptotic efficency}

What we have seen in Sections \ref{hm:sec_optest-1} and \ref{hm:sec_optest-2} concludes the following theorem.

\begin{thm}
Assume that the conditions \eqref{hm:g_A1} and \eqref{hm:g_A2}, and \eqref{hm:me_A3} to \eqref{hm:me_A4} hold.
Then, the random equation \eqref{hm:ME_random.eq} has a unique solution $(\bes^0,\ses^0)$ with $\pr$-probability tending to $1$, and we have \eqref{hm:theta0} for any norming matrix $\vp_{n}(\theta)$ satisfying \eqref{hm:vp-conditions} (without the uniformity in $\theta$).
Moreover, the estimator $\tes^1$ defined by \eqref{hm:1step.est} is asymptotically efficient in the sense that the asymptotic normality \eqref{hm:asymp.efficiency} holds.
\label{hm:thm_me}
\end{thm}

\medskip

Theorem \ref{hm:thm_me} remedies the incorrect uses of the delta method for the asymptotic normality results given in Sections 3.3 and 3.4 of \cite{Mas15LM}, together with the relevant references cited therein.


\subsection{Examples and simulations}\label{sec:sim}

Here we discuss the  two specific examples of $g$ treated in \cite{Mas15LM}:
\begin{equation}
\text{$g_{q}(y):=(|y|^{q}, |y|^{2q})^{\top}$ \quad and \quad $g_{\log}(y):=(\log|y|, \, (\log|y|)^{2})^{\top}$.}
\nonumber
\end{equation}
The choice $g_{q}$ is closely related to Todorov's result \cite{Tod13}, where he derived a stable central limit theorem associated with the power-variation statistics for a class of pure-jump It\^o semimartingales.

It suffices to set $\beta_{g}=0$ and $\beta_{g}=6q$ for $g_{\log}$ and $g_{q}$, respectively; note that the latter choice comes from the condition \eqref{hm:g_A2}.
Since $q>0$ should be given {\it a priori}, this implies that we are forced to pick a $q>0$ such that
\begin{equation}
q < \frac{\beta}{6}
\label{hm:add.eq1}
\end{equation}
with the true value $\beta\in(0,2)$ being unknown;
of course, this problem does not appear when, for example, we assume from the very beginning that $X$ has a finite mean, so that we may set $\beta_{g}=1$ and any $q<1/6$ can be used.
In Section \ref{hm.sec_verification.1} we will first verify the aforementioned conditions \eqref{hm:me_A3} to \eqref{hm:me_A4} for both of $g_{q}$ and $g_{\log}$ in parallel, with supposing that the tuning parameter $q$ satisfies \eqref{hm:add.eq1} for $g_{q}$.
Then in Section \ref{hm.sec_verification.2} we will show that
by introducing an auxiliary estimator of $\beta$ through splitting data we can proceed as if the condition \eqref{hm:add.eq1} does hold without any prior knowledge about $\beta\in(0,2)$.
\footnote{We are grateful for the anonymous reviewer for suggestion on this point.}

\subsubsection{Verification with knowing \eqref{hm:add.eq1} for $g_{q}$}
\label{hm.sec_verification.1}

The estimating equation \eqref{hm:ME_random.eq} can be successfully solved with respect to $\beta$,
and appropriate delta methods leads to the asymptotic normality of $\sqrt{n}(\bes^0-\beta)$; see \cite[Sections 3.3.5 and 3.3.6]{Mas15LM} for details.
The next step, estimation of $\sig$, is of primary importance in the present study.
Plugging $\bes^0$ into the first components of the original estimating equations (corresponding to $\log|y|$ and $|y|^{q}$ respectively for $g_{\log}$ and $g_{q}$),
and then again using the tightness of $\{\mbbs_{n}(\beta,\sig)\}_{n}$, we deduce the following stochastic expansions:
\begin{align}
\frac{\sqrt{n}}{\log(1/h)}( \log\ses^0 - \log\sig ) &= \beta^{-2}\sqrt{n}(\bes^0 -\beta ) + O_{p}\bigg(\frac{1}{\log(1/h)}\bigg), \nn\\
\frac{\sqrt{n}}{\log(1/h)}\{ (\ses^0)^{q} - \sig^{q} \} &= -q\beta^{-2}\sig^{q}\sqrt{n}(\bes^0 -\beta ) + O_{p}\bigg(\frac{1}{\log(1/h)}\bigg),
\nonumber
\end{align}
for the cases of $g_{\log}$ and $g_{q}$, respectively;
these expansions clarify asymptotic linear dependence of $\bes^{0}$ and $\ses^0$, which is quite similar to \cite[Eq.(3.76)]{Mas10proc} in the context of moment estimation of a skewed stable model without drift.
Thus, the tightness condition \eqref{hm:me_A1} holds in both cases.

We know that
\begin{align}
\mbbg_{0}(\rho) &=
\left(
\begin{array}{c}
\mathfrak{C}(1/\beta-1)+\log\sig \\
\pi^{2}(1/\beta^{2}+1/2)/6 + \{\mathfrak{C}(1/\beta-1)+\log\sig\}^{2}
\end{array}
\right),
\nn\\
\mbbg_{0}(\rho) &= 
\left(
\begin{array}{c}
\sig^{q}C(\beta,q) \\
\sig^{2q}C(\beta,2q)
\end{array}
\right)
\nonumber
\end{align}
for the cases of $g_{\log}$ and $g_{q}$, respectively, where
\begin{equation}
C(\beta,q):=C_{q}\Gam\bigg(1-\frac{q}{\beta}\bigg)
\nonumber
\end{equation}
with $C_{q}:=2^{q}\pi^{-1/2}\Gam\left((q+1)/2\right)/\Gam(1-q/2)$: the differentiability \eqref{hm:me_A3} is trivial.
Solving the random equation $\mbbg_{n}(\beta') = \mbbg_{0}(\beta',\sig')$ is numerically easy and fast:
we have the explicit solution in the $g_{\log}$ case, and also we can apply the simple bisection search for the $g_{q}$ case.

For any random sequence $\tilde{\beta}_{n}$, straightforward computations lead to the following:
\begin{align}
& -\p_{\beta}\mbbg_{n}(\tilde{\beta}_{n}) \nn\\
&= \left(
\begin{array}{c}
\ds{\frac{1}{\tilde{\beta}_{n}^{2}}\log(1/h)} \\
\ds{\frac{2}{\tilde{\beta}_{n}^{2}}\log(1/h)\frac{1}{2k}\sum_{j\ne k+1}\log\big| h^{-1/\beta}(\D_{j}X-\hat{\mu}^{0}_{n}h ) \big| 
+ \frac{2}{\tilde{\beta}_{n}^{2}}\{\log(1/h)\}^{2}
\bigg(\frac{1}{\tilde{\beta}_{n}}-\frac{1}{\beta}\bigg)}
\end{array}
\right) \label{hm:me-A2-1}
\end{align}
for the $g_{\log}$ case, and also
\begin{align}
-\p_{\beta}\mbbg_{n}(\tilde{\beta}_{n}) &= 
-\frac{q}{\tilde{\beta}_{n}^{2}}\log(1/h) \left(
\begin{array}{cc}
h^{q(1/\beta-1/\tilde{\beta}_{n})} & 0 \\
0 & 2h^{2q(1/\beta-1/\tilde{\beta}_{n})}
\end{array}
\right) \mbbg_{n}(\beta)
\label{hm:me-A2-2}
\end{align}
for the $g_{q}$ case.
Let $\tilde{\beta}_{n}$ be a $\sqrt{n}$-consistent estimator of $\beta$, and recall that $\mbbg_{n}(\beta)=\mbbg_{0}(\beta,\sig)+O_{p}(1/\sqrt{n})$.
Then, the right-hand sides of \eqref{hm:me-A2-1} and \eqref{hm:me-A2-2} both equal $\sig\beta^{-2}\log(1/h) \p_{\sig}\mbbg_{0}(\beta,\sig) + o_{p}(1)$, hence the condition \eqref{hm:me_A2}.

Finally, the non-degeneracy condition \eqref{hm:me_A4} can be verified by direct computations as in \cite{Mas15LM}:
indeed, the determinant $|\Gam_{0}(\rho)|$ equals $-\pi^{2}/(3\beta^{3})$ for the $g_{\log}$ case,
and $2q^{2}\beta^{-2}\sig^{3q}C_{q}C_{2q}\{\p\Gam(1-2q/\beta)-\p\Gam(1-q/\beta)\}$ for the $g_{q}$ case, both non-vanishing.

Thus we have verified the conditions \eqref{hm:me_A3} to \eqref{hm:me_A4} for $g_{\log}$, and also for $g_{q}$ when we know \eqref{hm:add.eq1} beforehand.

\subsubsection{Verification without knowing \eqref{hm:add.eq1}}
\label{hm.sec_verification.2}


Now let us consider $g_{q}$ when we do not know if \eqref{hm:add.eq1} holds or not.
Pick an increasing sequence $(m_{n})\subset\mbbn$ such that
\begin{equation}
m=m_{n}\to\infty \quad \text{and} \quad m=o(n), \qquad n\to\infty.
\nonumber
\end{equation}
Let $\hat{\beta}_{\log,m}$ denote the $g_{\log}$-moment estimator of $\beta$ constructed only from the first $m$ increments $\D_{1}X,\dots,\D_{m}X$.
As was seen in Section \ref{hm.sec_verification.1} we have $\hat{\beta}_{\log,m} \cip \beta$.

Let $\ep\in(0,1)$ be a (small) number, and set
\begin{equation}
\hat{q}_{m} = \frac{1}{6}(1-\ep)\hat{\beta}_{\log,m}.
\nonumber
\end{equation}
Then we observe that
\begin{align}
\pr\bigg( \hat{q}_{m} < \frac{\beta}{6} \bigg) \ge \pr\bigg( \big|\hat{\beta}_{\log,m}-\beta\big| < \frac{\ep\beta}{1-\ep} \bigg) \to 1.
\label{hm:add.eq3}
\end{align}
Now, denote by $\hat{\mu}^{\perp}_{n-m}$ the sample-median based estimator of $\mu$ defined as in \eqref{hm:smed_mu0} except that it is now constructed from $\D_{m+1}X,\dots,\D_{n}X$;
here again, we may set $n-m$ odd.
Also, denote by $(\hat{\beta}^{\perp}_{n-m}(q),\, \hat{\sig}^{\perp}_{n-m}(q))$ the $(q,2q)$th-moment estimator of $(\beta,\sig)$ defined as in Section \ref{hm.sec_verification.1} except that it is computed based on $\D_{m+1}X,\dots,\D_{n}X$ with plugging-in $\hat{\mu}^{\perp}_{n-m}$. Then, we define
\begin{equation}
\hat{\theta}^{\perp}_{n-m} := \left( \hat{\beta}^{\perp}_{n-m}(\hat{q}_{m}),\, \hat{\sig}^{\perp}_{n-m}(\hat{q}_{m}),\, \hat{\mu}^{\perp}_{n-m}\right).
\nonumber
\end{equation}
By its construction, $\hat{q}_{m}$ is independent of $(\D_{m+1}X,\dots,\D_{n}X)$.
Further, since $(\D_{m+1}X,\dots,\D_{n}X)$ and $(\D_{1}X,\dots,\D_{n-m}X)$ have the same distribution, the distribution of the estimator $\hat{\theta}^{\perp}_{n-m}$ conditional on the event $\{\hat{q}_{m}<\beta/6\}$ is the same as the one of $\hat{\theta}^{\perp}_{n-m}$ based on $(\D_{1}X,\dots,\D_{n-m}X)$ where we can regard $(\hat{q}_{m})$ as a constant sequence converging to $(1-\ep)\beta/6$ as $n\to\infty$.
Hence, under the additional condition
\begin{equation}
\left| \vp_{n-m}^{-1}(\theta)\,\vp_{n}(\theta) - I_{3}\right| \to 0,
\label{hm:add.eq4}
\end{equation}
it follows from \eqref{hm:add.eq3} that
\begin{equation}
\vp_{n}(\theta)^{-1}(\hat{\theta}^{\perp}_{n-m} -\theta) = O_{p}(1).
\nn
\end{equation}
That is, $\hat{\theta}^{\perp}_{n-m}$ as an initial estimator meets the condition \eqref{hm:theta0}.
This together with Theorem \ref{hm:thm_me} and what we have seen in Section \ref{hm.sec_verification.1} concludes that
the estimator $\tes^1$ given by \eqref{hm:1step.est} with $\tes^0 = \hat{\theta}^{\perp}_{n-m}$ is asymptotically efficient.
Thus, we have shown that under the condition \eqref{hm:add.eq4}, the $(q,2q)$th-moment estimator can be used as if \eqref{hm:add.eq1} is true without any previous knowledge of $\beta\in(0,2)$;
in practice, we first compute $\hat{\beta}_{\log,m}$ and then fix any $\hat{q}_{m} < \hat{\beta}_{\log,m}/6$, followed by $\hat{\theta}^{\perp}_{n-m}$ for that $\hat{q}_{m}$.

Finally, we note that the condition \eqref{hm:add.eq4} holds for both \eqref{phibeta} and \eqref{phis} as soon as
\begin{equation}
\frac{h_{n-m}}{h_{n}} \to 1 \quad \text{and} \quad \frac{\log h_{n}}{\log h_{n-m}} \to 1.
\nonumber
\end{equation}
These conditions and \eqref{hm:liminf.p} hold if, for example, $h_{n}=n^{-\kappa}$ for $\kappa\in(0,1]$.


\subsubsection{Simulation}

Finally we present a simulation result.
We set $\mu=0$ for brevity and focused on comparing performances of joint estimation of $(\beta,\sig)$ through the MLE, the one-step MLE, and the $(q,2q)$th-moment estimator with $q=0.1$;
for simplicity, we implicitly assume that \eqref{hm:add.eq1} holds, that is, $\beta>0.6$.
We set $h=1/n$ and simulated $1,000$ Monte-Carlo samples of $n=2^9=512$ i.i.d. symmetric $\beta$-stable random variables
$\D_{j}X$ with index $\beta=1.6$ and scale $\sig h^{1/\beta}$ for $\sig=1.2$.
The $(q,2q)$th-moment estimator was also used as the initial estimator in the one-step MLE \eqref{hm:1step.est} with rate matrix being the upper-left $2\times 2$-submatrix of \eqref{phibeta}.

Figure \ref{fig:SaSHF} shows the histograms based on the normalized statistical errors
\begin{equation}
\text{$\sqrt{n} ( \hat{\beta}_n - \beta )$ \quad and \quad 
$ \frac{\sqrt{n}}{ \beta^{-2} \sigma \log (1/ h) } ( \hat{\sigma}_n - \sigma )$}
\nonumber
\end{equation}
on the left and the right, respectively, for $(\bes,\ses)$ being the simulated sequences of the moment estimators (ME, on the top), the one-step MLEs (in the center), and the MLEs (on the bottom).
In each panel, the solid line represents the asymptotic normal distribution with efficient variance
$\Sigma_{22}(\Sigma_{11}\Sigma_{22}-\Sigma_{12}^{2})^{-1}$; see \eqref{hm:lb-beta} and \eqref{hm:lb-sigma}.
The implementations of the likelihood, the score and the Fisher information for computing the sequences of the MLE and the one-step MLE are based on \cite{MatTak06}.
From Figure \ref{fig:SaSHF}, we can observe that the histograms of the moment estimators are far from the efficient asymptotic normal distributions, whereas both of the one-step MLE and the MLE sequences show much better performances.

\begin{figure}[h]
\centering
\includegraphics[
scale=0.7
]{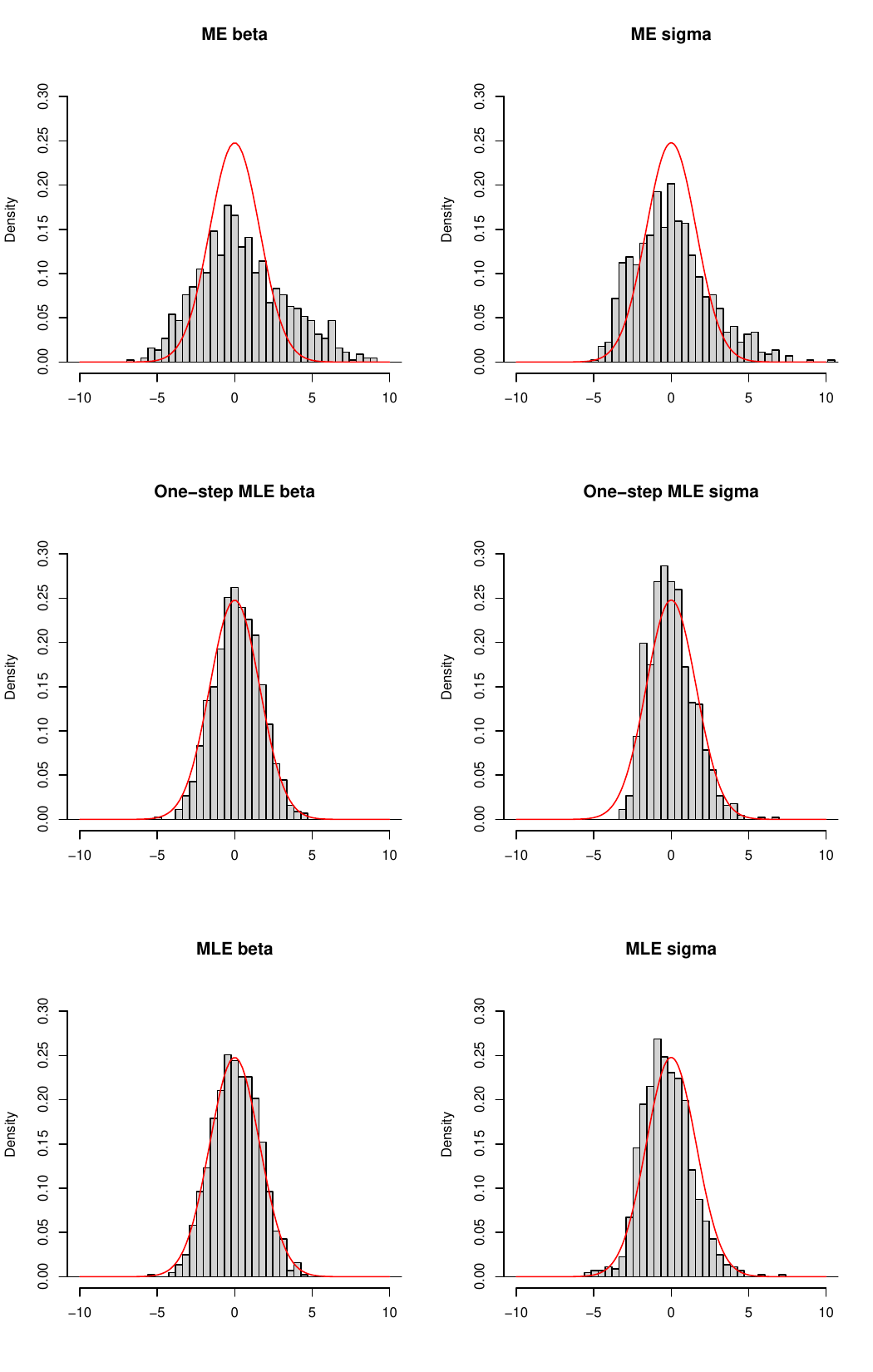}
\caption{Histograms of the MLE, one-step MLE, and the $(q,2q)$th-moment estimator (ME) with $q=0.1$ for $(\beta,\sig)$, based on $1,000$ independent simulated paths of sample size $n=2^{9}=512$.
The solid lines represent the asymptotic normal distribution with the efficient variance specified in \eqref{hm:lb-beta} and \eqref{hm:lb-sigma}.}
\label{fig:SaSHF}
\end{figure}

\bigskip

\noindent
\textbf{Acknowledgement.}
The authors also thank the reviewers and the associated editor for their valuable comments, which in particular led to substantial improvements of the arguments in Sections \ref{hm:sec_optest-2} and \ref{sec:sim}.
HM especially thanks Professor Jean Jacod for letting him notice the mistake in \cite{Mas09jjss}, which has been fixed in the present paper.

\bigskip

\def\cprime{$'$} \def\polhk#1{\setbox0=\hbox{#1}{\ooalign{\hidewidth
  \lower1.5ex\hbox{`}\hidewidth\crcr\unhbox0}}} \def\cprime{$'$}
  \def\cprime{$'$}

\end{document}